\documentclass[a4paper,11pt]{article}
\usepackage{graphicx} 

\usepackage{amsfonts}
\usepackage{amssymb,amsthm,amsmath,graphicx,bm,ebezier,amsthm}
\usepackage{url}
\usepackage[dvipsnames]{xcolor}

\usepackage[ruled,vlined,linesnumbered]{algorithm2e}
\usepackage[noend]{algpseudocode}
\usepackage{enumitem}
\usepackage{hyperref}

\newtheorem{theorem}{Theorem}
\newtheorem{example}[theorem]{Example}
\newtheorem{definition}[theorem]{Definition}
\newtheorem{proposition}[theorem]{Proposition}
\newtheorem{corollary}[theorem]{Corollary}
\newtheorem{lemma}[theorem]{Lemma}

\def\CaH{\mathcal{H}}
\def\CaC{\mathcal{C}}
\def\CaA{\mathcal{A}}

\def\N{\mathbb{N}}
\def\Z{\mathbb{Z}}

\def\k{\mathbb{K}}
\def\msg{\operatorname{msg}}

\def\g{\operatorname{g}}

\def\Fb{\operatorname{Fb}}

\def\Ap{\operatorname{Ap}}
\def\Mult{\operatorname{Mult}}
\def\supp{\operatorname{supp}}

\title{A note on strong affine semigroups}
\date{}
\author{I. García-Marco,
R. Tapia-Ramos,
and A. Vigneron-Tenorio
}

\begin{document}

\maketitle

\abstract{
This work introduces and studies strong affine semigroups, extending the notion of strong numerical semigroups to the higher-dimensional setting.
We show that non-numerical strong affine semigroups present structural differences with respect to strong numerical semigroups.
Special attention is devoted to strong $\CaC$-semigroups. We prove that the family of strong $\CaC$-semigroups with a given set of multiplicities $E$ admits a maximal element and has a tree structure. We characterize when this family is finite and provide an algorithm to compute all such semigroups up to a fixed genus. We also introduce the notion of special strong affine semigroups and obtain refined versions of several previous results.
Finally, we study toric ideals arising from strong affine semigroups, determining their indispensable monomials and Betti elements for several families.

{\small

}

\section{Introduction}
Affine semigroups provide a natural higher-dimensional generalization of numerical semigroups and constitute a central object in combinatorial commutative algebra and discrete geometry.
Let $p$ be a positive integer and let $\N$ be the set of non-negative integers. An affine semigroup $S\subseteq \N^p$ is a finitely generated submonoid of $\N^p$, that is, there exists a finite set $L = \{l_1,\ldots,l_s\} \subset \mathbb{N}^p$ such that $S$ is equal to $\langle L \rangle :=\{\sum_{i=1}^s \alpha_i l_i \mid \alpha_1,\ldots , \alpha_s\in \N\}$. It is well-known that every affine semigroup admits a unique minimal generating set (see \cite{libro_rosales2}), denoted by $\msg(S)$.

A notable subset of $\msg(S)$ is the set of multiplicities of $S$, denoted $\Mult (S)$. This set consists of the minimal non-zero elements in $S$ with respect to the partial order induced by its spanning cone. Hence, we can write $\msg(S) = E \sqcup A$ with $E=\Mult (S)$, and $A=\msg(S)\setminus E$. We say that $S=\langle E\sqcup A\rangle\subseteq\N^p$ is strong if for all $x,y\in S\setminus\{0\}$ with $x\neq y$, there exists $e\in E$ such that $x+y-e\in S$. This definition is inspired by \cite{Strong nuevo de JC}, which is formulated for numerical semigroups. Originally, strong numerical semigroups were introduced in \cite[Section 5]{RoblesRosalesModularFrob}.

This paper aims to study strong affine semigroups, which exhibit several fundamental differences with respect to the numerical semigroup case. While strong numerical semigroups are relatively rare, strong affine semigroups are abundant. Indeed, every affine semigroup with $\msg(S) = \Mult(S)$ is strong. Moreover, unlike the numerical case, the intersection of two strong affine semigroups with the same set of multiplicities need not be strong. Hence, in the language of \cite{PseudoFrobeniusVariety}, strong affine semigroups do not form a Frobenius pseudo-variety.

A relevant class of affine semigroups is the family of $\CaC$-semigroups, which are affine semigroups with finite complement in their spanning cone. The cardinality of this complement is called the genus of $S$ and is denoted by $\g(S)$.
Numerical semigroups are precisely one-dimensional $\mathcal{C}$-semigroups.
We study the family $\mathcal{ST}(E)$ of strong $\CaC$-semigroups with set of multiplicities $E$. We observe that this family admits a maximal element $R(E)$. Furthermore, we describe all the elements in $\mathcal{ST}(E)$ as vertex in a tree with root $R(E)$, and where semigroups at distance $d$ from $R(E)$ have genus $\g(R(E)) + d$.
 This construction yields an effective algorithm to compute all elements of $\mathcal{ST}(E)$ up to a fixed genus.
Interestingly, the set $\mathcal{ST}(E)$ may be finite or infinite depending on the choice of $E$. We characterize when either case happens. This contrasts with the numerical semigroup case, where the set of strong numerical semigroups with fixed multiplicity is always infinite.

We also introduce and study a subclass of strong affine semigroups, called special strong semigroups, closely related to affine semigroups of maximal embedding dimension.

Finally, we study the toric ideals associated with strong affine semigroups, with particular emphasis on indispensable monomials and Betti elements. Using combinatorial techniques, we obtain explicit descriptions for several families of strong and special strong semigroups.

The content of this work is organized as follows: in Section \ref{sec:Preli}, we provide the basic background on affine semigroups and fix the notation used throughout the work.
In Section \ref{sec:chra}, we present three different characterizations of strong affine semigroups, all of which extend the numerical case. In particular, we characterize the property of being strong by describing the Apéry set $\Ap(S,E)$, and we provide a criterion to determine when an affine semigroup is strong in terms of its minimal generators.
In Section~\ref{sec:ST}, we analyze the set $\mathcal{ST}(E)$ of strong $\CaC$-semigroups with set of multiplicities $E$, and we provide an algorithm to compute all such semigroups up to a given genus. This construction naturally leads to an associated tree structure. In Section~\ref{sec:strongMultGenus}, we characterize when $\mathcal{ST}(E)$ is finite. In Section~\ref{sec:sss}, we introduce the concept of special strong affine semigroups and specialize some results obtained so far. Finally, in Section~\ref{sec:indisp}, we describe the set of indispensable monomials and Betti elements of strong affine semigroups.

\section{Preliminaries}\label{sec:Preli}

This section introduces the notation and definitions that will be used throughout the paper. Given a positive integer $q\in \N$, we denote by $[q]$ the set $\{1,2, \ldots, q\}$.

For a finite set $L =  \{l_1,\ldots,l_k\} \subseteq \mathbb N^p$, the integer cone spanned by $L$  is
\[ \CaC(L) := \left\{ \sum_{i = 1}^k \alpha_i l_i \ \vert \ \alpha_i \in \mathbb R_{\geq 0} \right\} \cap \N^p.\]
For an affine semigroup $S$, we set $\CaC(S) := \CaC(\msg(S))$. By \cite[Corollary 2.10]{Bruns}, these integer cones are affine semigroups.

Let $S \subseteq \mathbb{N}^p$ be an affine semigroup, the partial order induced by $S$ is defined by  \[ x \le_S y  \ \Longleftrightarrow \ y - x \in S;\] for all $x,y \in \CaC(S)$. With this notion, one has that $\msg(S) = {\rm min}_{\le_S}(S\setminus\{0\})$, i.e., the minimal set of generators of $S$ coincides with the minimal elements in $S \setminus\{0\}$ with respect to $\le_S$.

For a numerical semigroup $S\subseteq\N$, its multiplicity is the minimum element in $S\setminus\{0\}$.
For an affine semigroup $S \subseteq \N^p$, we define its set of multiplicities as $\Mult (S):=\min _{\le_{\CaC(S)}} (S\setminus\{0\})$, i.e., the set of minimal elements with respect to the partial order induced by $\CaC(S)$. Since $S \subseteq \CaC(S)$, then $x \le_{\CaC(S)} y$ whenever  $x \le_S y$ and, hence, $\Mult(S) \subseteq \msg(S)$. So, for any affine semigroup $S$, we write $\msg(S)=E\sqcup A$, where $E=\Mult (S)$, and $A=\msg(S)\setminus E$.
In general, for a finite set $E\subset \N^p\setminus \{0\}$, we say that $E$ is $\CaC(E)$-incomparable, or simply incomparable, if $E=\min _{\le_{\CaC(E)}} (E)$.

The Apéry set of an affine semigroup $S\subseteq \N^p$ with respect to an element $m\in S\setminus\{0\}$ is defined as $\Ap(S,m)=\{s\in S\mid s-m\notin S\}$. If $S$ is a numerical semigroup, then $\Ap(S,m)$ is a finite set with several remarkable properties. However, in general, the Apéry set with respect to an element is not finite when $p>1$. For a finite set $L \subseteq S \setminus \{0\}$, we denote $\Ap(S,L) := \bigcap_{m\in L} \Ap (S,m)$ and we have that this set is finite if and only if $\CaC(L) = \CaC(S)$ (see, e.g., \cite[Theorem 2.6]{GGM}). In particular, we have that $\Ap(S, \Mult(S))$ is always finite.

An affine semigroup $S$ with finite complement in $\CaC(S)$ is called a $\CaC$-semigroup. From now on, we denote by $\CaH(S)=\CaC(S)\setminus S$ the set of gaps of $S$.

Given an finite incomparable set $E \subseteq \N^p$, consider the affine semigroup \[ R(E) := \{0\} \cup (E + \mathcal{C}(E)).\]
Every affine semigroup $S$ with $\Mult(S) = E$ satisfies that $\langle E \rangle \subseteq S  \subseteq R(E).$ Extending the notation given in \cite{Strong nuevo de JC}, we define $\mathcal{ST}(E)$ as the set of all strong $\CaC$-semigroups whose set of multiplicities is $E$. Since $R(E)$ is a strong $\CaC$-semigroup with set of multiplicities $E$, then $R(E)$ is the maximum element of $\mathcal{ST}(E)$ with respect to the inclusion.
Given a monomial order on $\N^p$ (see \cite{CoxLO} for a general discussion on monomial orders) and $S \in \mathcal{ST}(E)$ such that $S \neq R(E)$, we define the Frobenius element of $S$ with respect to $\preceq$ as $\Fb_{E}(S):=\max_{\preceq}(R(E)\setminus S)$. By convention  $\Fb_{E}(R(E))=(-1,\ldots ,-1)$.

\section{Characterizing strong affine semigroups}\label{sec:chra}
Recall that an affine semigroup $S=\langle E\sqcup A\rangle\subseteq\N^p$ is strong if for all $x,y\in S\setminus\{0\}$ with $x\neq y$, there exists $e\in E$ such that $x+y-e\in S$.  Following an argument analogous to the one used in the proof of \cite[Proposition 2]{Strong nuevo de JC}, we present the following characterization.

\begin{lemma}
Let $S=\langle E\sqcup A\rangle\subseteq \N^p$ be an affine semigroup. Then, $S$ is strong if and only if, for all $x,y\in S\setminus\{0\}$ such that $x- y \neq km$ for all $(m,k)\in E\times \N$, there exists $e\in E$ such that $x+y-e\in S$.
\end{lemma}

Given $S$ an affine semigroup, then $S\cap \tau$ is isomorphic to a numerical semigroup for any ray $\tau$ such that $\{0\} \subsetneq \tau \cap \CaC(S)$ (see  \cite[Lemma 2]{check_C_semigr}). Indeed, since $\tau\cap\N^p=\langle n \rangle$ for some $ n\in\N^p$, then $S\cap \tau$ is isomorphic to $\{k\in\N\mid kn\in S\}$, which is a submonoid of $\N$.

Let us prove that the property of being strong is inherited by the affine semigroups associated to the faces and, in particular, by the numerical semigroups associated to the extremal rays. Note that any integer cone $\CaC$ is an integer polyhedron determined by a set of supporting hyperplanes. We assume that the cone corresponds to the positive space delimited by these hyperplanes, specifically
\begin{equation}\label{coneHyperplanes}
\CaC=\{x\in \N^p\mid h_1(x)\ge 0,\ldots ,h_q(x)\ge 0\},
\end{equation}
where the coefficients of each supporting hyperplane are integers relatively prime. Each $d$-dimensional face $\sigma$ of $\CaC$ is determined by the intersection of $p-d$ supporting hyperplanes of $\CaC$ (\cite{Bruns}). An extremal ray $\tau$ of $\CaC$ is a $1$-dimensional face.

\begin{lemma}
Let $S=\langle E\sqcup A\rangle\subseteq \N^p$ be a strong affine semigroup. Then, $S \cap \sigma$ is strong for every face $\sigma$ of $\CaC$.
\end{lemma}

\begin{proof} Let $\sigma$ be a face of $\CaC$. Observe that $\Mult(S \cap \sigma) = E \cap \sigma$.
Consider a supporting hyperplane $h$ of $\CaC(S)$ defining $\sigma$, i.e.,
\begin{itemize} \item $h(z) \geq 0$ for all $z \in \CaC(S)$, and
\item  $h(z) = 0$ if and only if $z \in \sigma$, for all $z \in \CaC(S)$.
\end{itemize}
Consider two distinct non-zero elements $x,y\in S \cap \sigma$. Since $S$ is strong, there exists an $e \in E$ such that $x+y-e \in S$. Then, we have that $0 \leq h(x+y-e) = h(x) + h(y) - h(e) = -h(e)$ (since $x+y-e \in S$) and $h(e) \geq 0$ (since $e \in S$). Thus, $h(e) = 0$ and we conclude that $e \in \sigma$. As a consequence, $e \in E \cap \sigma$ and $x+y-e \in S \cap \sigma$.
\end{proof}

The next result generalizes \cite[Proposition 5.7]{RoblesRosalesModularFrob}, originally stated for numerical semigroups. In addition, while in the numerical case only one implication is proved, the next result is an equivalence, which in particular recovers the numerical setting.

\begin{lemma}\label{lemaAperys}
Let $S=\langle E\sqcup A\rangle\subseteq \N^p$ be an affine semigroup. Then, $S$ is a strong affine semigroup if and only if $\Ap (S,E)= A \sqcup B\sqcup \{0\}$, for some $B\subseteq 2A=\{2a\mid a\in A \}$.
\end{lemma}

\begin{proof}
Assume that $S$ is a strong affine semigroup. Trivially, $A\sqcup \{0\}\subseteq \Ap (S,E)$. Let $x\in \Ap (S,E)\setminus A$, and suppose $x=s+s'$ for some $s,s'\in S\setminus \{0\}$. If $s\neq s'$, then by hypothesis there exists $e\in E$ such that $x-e=s+s'-e\in S$, and thus $x\notin \Ap (S,E)$. So, $s=s'$ and $x=2s$ with $s\in \msg (S)$. Since $x\in \Ap (S,E)$, then $s\in A$, which completes the proof.

Conversely, assume that $\Ap(S,E)= A \sqcup B \sqcup \{0\}$ with $B\subseteq 2A$, and let $x,y\in S\setminus\{0\}$ such that $x\ne y$. So, $x=\sum_{e\in E}\alpha_e e + \sum _{a\in A} \beta_a a$, and $y=\sum_{e\in E}\alpha'_e e + \sum _{a\in A} \beta'_a a$ where $\alpha_e,\alpha'_e,\beta_a,\beta'_a\in \N$. If $\sum_{e\in E} (\alpha_e+\alpha'_e)\neq 0$, then $x+y-e\in S$ for some $e\in E$. Now, assume $\sum_{e\in E} (\alpha_e+\alpha'_e)= 0$. Since $x\ne y$, we distinguish two different cases:
\begin{itemize}
    \item The sum $x+y$ equals $(\beta_a+\beta'_a)a$ with $\beta_a+\beta'_a\geq3$ for some $a\in A$. Hence, $x+y=3a+(\beta_a+\beta'_a-3)a$. Observe that $3a\notin \Ap (S,E)$ and thus $3a-e\in S$ for some $e\in E$, which implies that $x+y-e\in S$.
    \item There exist distinct $a_1,a_2\in A$ such that $\beta_{a_1}, \beta'_{a_2}\geq 1$. Since $a_1+a_2\notin \Ap(S,E)$, there exists $e\in E$ such that $a_1+a_2-e\in S$, and it follows that $x+y-e\in S$.
\end{itemize}
We conclude that $S$ is strong.
\end{proof}

\begin{corollary}\label{escritura_unica}
Let $S$ be a strong affine semigroup. Then, any element in $\Ap (S,E)$ has a unique expression in terms of the minimal generators of $S$.
\end{corollary}

\begin{proof}
By Lemma \ref{lemaAperys}, it is sufficient to prove the statement for the elements in $A \sqcup B\sqcup \{0\}$, for some $B\subseteq 2A$. Note that it is true for every $a\in A \sqcup \{0\}$. Consider $2a \in B$, and assume that $2a=x+y$ with $x,y\in S$. Since $S$ is strong, if $x\neq y$, then $2a\notin \Ap(S,E)$. Hence, $x=y$, and then $a=x$. Thus, $2a$ has a unique expression.
\end{proof}

An analogue of \cite[Proposition 5.1]{RoblesRosalesModularFrob} holds in the context of affine semigroups, giving us a method to check whether a given affine semigroup is strong.

\begin{proposition}\label{CheckStrongByGenerators}
Let $S=\langle E\sqcup A\rangle\subseteq \N^p$ be an affine semigroup. Then, $S$ is a strong semigroup if and only if for all $a_1,a_2\in A$ with $a_1\neq a_2$, $a_1+a_2-e_1,3a_1-e_2\in S$ for some $e_1,e_2\in E$.
\end{proposition}

\begin{proof}
The direct implication is immediate. So, assume that for all $a_1,a_2\in A$ with $a_1\neq a_2$, $a_1+a_2-e_1,3a_1-e_2\in S$ for some $e_1,e_2\in E$. Consider $x,y\in S\setminus \{0\}$ such that $x\neq y$. Therefore, $x=\sum_{e\in E}\alpha_e e + \sum _{a\in A} \beta_a a$, and $y=\sum_{e\in E}\alpha'_e e + \sum _{a\in A} \beta'_a a$ with $\alpha_e,\alpha'_e,\beta_a,\beta'_a\in \N$. If $\sum_{e\in E} (\alpha_e+\alpha'_e)\neq 0$, then $x+y-e\in S$ for some $e\in E$. Otherwise, if $\sum_{e\in E} (\alpha_e+\alpha'_e)=0$, consider $X=\{a\in A\mid x-a\in S\}$ and $Y=\{a\in A\mid y-a\in S\}$. Since both $x$ and $y$ are non-zero, $X$ and $Y$ are not empty. Now, we consider two cases:
\begin{itemize}
\item If $X=Y=\{a\}$, then $x+y=(\beta_a+\beta'_a)a$. Since $x\neq y$, it follows that $\beta_a+\beta'_a\ge 3$, and thus $x+y-e\in S$ for some $e\in E$.
\item If $X\neq Y$. Let $a_1\in X$ and $a_2\in Y$ be distinct elements. By hypothesis, $a_1+a_2- e\in S$ for some $e\in E$. Then,
\begin{multline*}
x+y-e= \sum _{a\in A\setminus\{a_1\}} \beta_a a + \sum _{a\in A\setminus\{a_2\}} \beta'_a a +\\
(\beta_{a_1}-1)a_1+(\beta'_{a_2}-1)a_2+a_1+a_2-e\in S.
\end{multline*}
\end{itemize}
In any case, $x+y-e\in S$, for some $e\in E$. Therefore, $S$ is strong.
\end{proof}

\section{Strong $\CaC$-semigroups with a given set of multiplicities}\label{sec:ST}
Fix an incomparable set $E\subset \N^p\setminus\{0\}$ and let $\CaC=\CaC(E)$.
The objective of this section is to compute the set $\mathcal{ST}(E)$.
The strategy is as follows. Starting from the maximal element of $\mathcal{ST}(E)$ with respect to inclusion, all strong $\CaC$-semigroups with set of multiplicities $E$ are obtained by iteratively removing suitable elements. A key point is to characterize when this method,  which increases the genus by one, preserves the property of being strong.

Recall that $R(E)=\{0\}\cup (E + \CaC(E))$ is the maximal element in $\mathcal{ST}(E)$ with respect to the inclusion. To simplify, we write $\CaC=\CaC(E)$, and we assume that $\CaC$ has, at least, $p$ extremal rays. If $\CaC$ has only $p' < p $ extremal rays, then it is contained in a subspace of dimension $p'$, and $R(E)$ can be uniquely identified with a $\CaC'$-semigroup with $p'$ extremal rays in
$\mathbb{N}^{p'}$. Thus, this situation is already covered by taking $p = p'$.

To describe the elements of $\mathcal{ST}(E)$, in analogy with properties of strong numerical semigroups established in \cite{RoblesRosalesModularFrob}, we present the following result. Recall that $\Fb_E(S)=\max_\preceq(R(E)\setminus S)$.

\begin{lemma}\label{strongUnionFrob}
Let $\preceq$ be a monomial order on $\N^p$, and let $S\neq R(E)$ be a strong $\CaC$-semigroup, then $S \cup \{\Fb_E(S)\}$ is a strong $\CaC$-semigroup.
\end{lemma}

\begin{proof}
Trivially, $S \cup \{\Fb_E(S)\}$ is a $\CaC$-semigroup. Besides, if $S \cup \{\Fb_E(S)\}=R(E)$, then it is also strong. So, assume $S \cup \{\Fb_E(S)\}\neq R(E)$. Consider $S$ is minimally generated by the disjoint union $E\sqcup A$, and let $x,y\in S \cup \{\Fb_E(S)\}$ with $x\neq y$.
If $x,y\notin\{\Fb_E(S)\}$, since $S$ is strong, there exists  $e\in E$ such that $x+y-e\in S\subset S\cup \{\Fb_E(S)\}$.
Otherwise, without loss of generality, assume that $x=\Fb_E(S)$. By definition of $E$, there exists $e\in E$ such that $y-e\in \CaC$.
Moreover, $\Fb_E(S)\prec \Fb_E(S) +y -e\in R(E)$, which implies that $\Fb_E(S) +y -e\in S\cup \{\Fb_E(S)\}$.
\end{proof}

Note that, given $S \in \mathcal{ST}(E)$ minimally generated by $\langle E \sqcup A \rangle$ and $x \in S \setminus \{0\}$ then $S \setminus \{x\}$ is a $\mathcal{C}$-semigroup if and only if $x \in \msg(S)$. Moreover, $\Mult(S \setminus \{x\} )=E$ if and only if $x \in A$. The following result shows a containment relation for $\msg(S \setminus \{x\}) \setminus E$, for $x \in A$.

\begin{lemma}\label{SistMin_S_menos_a}
Let $S=\langle E\sqcup A\rangle\subseteq \N^p$ be an affine strong semigroup, and $a\in A$. Then, the set $\msg (S\setminus\{a\})\setminus E$ is a subset of $(A\setminus\{a\})\cup \{2a\} \cup (\{a\}+E)$.
\end{lemma}

\begin{proof}
Consider $\tilde{A}=A\setminus\{a\}$. By \cite[Lemma 3]{Csemigroup}, the minimal system of generators of $S\setminus\{a\}$ is a subset of $E\cup \tilde{A}\cup \{2a,3a\} \cup \left(\{a\}+E\right)\cup (\{a\}+\tilde{A})$. Besides, note that $\Mult (S\setminus\{a\})=\Mult(S)$. Let
$\tilde{a}\in \tilde{A}$, since $S$ is strong, $a+\tilde{a}-e\in S$ for some $e \in E$. In particular, $a+\tilde{a}-e\notin\{0, a\}$. So, $a+\tilde{a}$ is not a minimal generator of $S\setminus\{a\}$. Analogously, $3a-e\in S\setminus\{a\}$ for some $e\in E$. Hence, $3a$ is not a minimal generator.
\end{proof}

The next result presents a criterion to decide whether $S\setminus\{a\}$ belongs to $\mathcal{ST}(E)$, where $S \in \mathcal{ST}(E)$ and $a \in A$. Note that its conditions can be tested since the minimal generating set of any affine semigroup is finite.

\begin{lemma}\label{SmenosXstrong}
Let $S=\langle E\sqcup A\rangle$ be an affine strong semigroup, and $a\in A$. The semigroup $S\setminus \{a\}$ is a strong semigroup if and only if:
\begin{enumerate}[label={(C\arabic*)}]
\item \label{cond1} For any $e\in E$ such that $a+e= a_1+a_2$ for two distinct elements $a_1,a_2\in A$, then $a+e-e_1\in S$ for some $e_1\in E\setminus\{e\}$.

\item \label{cond2} For any $e\in E$ such that $a+e= 3a_1$ with $a_1\in A$, then $a+e-e_1\in S$ for some $e_1\in E\setminus\{e\}$.
\end{enumerate}
\end{lemma}

\begin{proof}
By Lemma \ref{SistMin_S_menos_a}, the set $A'=\msg (S\setminus\{a\})\setminus E$ is a subset of $(A\setminus\{a\})\cup \{2a\} \cup (\{a\}+E)$.
The semigroup $S\setminus\{a\}$ is strong if and only if the set $A'$ satisfies the conditions stated in Proposition \ref{CheckStrongByGenerators}. Let $a_1,a_2\in A\setminus\{a\}$ and $x,y\in \{a\}+ E$ with $a_1\neq a_2$ and $x\neq y$, since $S$ is strong there exist $e_1,\ldots, e_9 \in E$ such that the elements
\[
\begin{array}{rrr}
a_1+a_2-e_1, & 2a+2a-e_2, & x+y-e_3, \\
a_1+2a-e_4, & a_1+x-e_5, & 2a+x-e_6, \\
3a_1-e_7, & 3(2a)-e_8, & 3x-e_9
\end{array}
\]
belong to $S$.
Since $E \sqcup A$ is the minimal generating set of $S$, none of the above elements can be equal to $a$, except $a_1+a_2-e_1$ and $3a_1-e_7$.
So, it remains to analyse the uncovered cases.
Suppose $a_1 + a_2 - e_1 = a$ and/or $3a_1-e_7=a$. If $S \setminus \{a\}$ is strong, then there exist $e,e'\in E$ such that $a_1 + a_2 - e,3a_1-e'\in S\setminus\{a\}$, and Conditions \ref{cond1} and \ref{cond2} hold.
Conversely, assume that Conditions~\ref{cond1} and~\ref{cond2} are satisfied. Then all the required combinations of elements in $A'$ verify the hypothesis of Proposition~\ref{CheckStrongByGenerators}. Therefore, $S\setminus\{a\}$ is strong.
\end{proof}

In particular, the conditions stated in Lemma \ref{SmenosXstrong} may not be satisfied, or the set $A$ might be empty. The next two examples illustrate both cases, respectively.

\begin{example}\label{EjemploArbolNoInfinito}
Consider the $\CaC$-semigroup $S$ minimally generated by the disjoint union of
\[
E=\{(1,7), (1,8), (2,2), (6,1), (7,1)\},
\]
and
\begin{multline*}
A=\{
(2,8), (2,9), (2,10), (2,11), (2,12), (2,13), (3,3), (3,4), (3,5), \\ (3,6), (3,7), (3,8), (3,11), (3,12), (3,13), (3,14), (4,3), (4,5), (4,6),\\ (4,7), (4,8), (4,9), (5,3), (5,4), (6,3), (6,4), (7,2), (7,3), (7,4), (8,2),\\ (8,4), (9,2), (10,2), (10,3), (11,2), (11,3), (12,3) \}.
\end{multline*}
By applying Proposition \ref{CheckStrongByGenerators}, it follows that $S$ is strong.
Notice that $S\setminus \{ (4,5)\}$ is not strong, since it does not verify the conditions of Lemma \ref{SmenosXstrong}, specifically observe that $(4,5)+ (2,2)=(3,3)+(3,4)$, but $(3,3)+(3,4)-m\notin S$ for every $m\in E \setminus \{ (2,2)\}$.
\end{example}

\begin{example}\label{EjemploAvacio}
Let $E=\{(1,1), (1,2), (3,1), (4,1)\}$ be an incomparable set, the strong $\CaC$-semigroup $R(E)$ is minimally generated by $E$. Hence, $A=\emptyset$, and thus there is no strong $\CaC$-semigroup strictly contained in $R(E)$.
\end{example}

Given an incomparable set  $E\subset\N^p$, we can provide a procedure to describe all the elements of $\mathcal{ST}(E)$, through an associated tree rooted in $R(E)$. We define this tree $G(E)$ as the graph with vertex set  $\mathcal{ST}(E)$, and the pair $(S,T)\in \mathcal{ST}(E)^2$ is an edge  if $T=S\cup \{\Fb_E(S)\}$. Equivalently, the sons of any $S\in\mathcal{ST}(E)$ are the semigroups $S\setminus\{x\}\in\mathcal{ST}(E)$, with $x\in\msg (S)\setminus E$ and $\Fb_E(S)\prec x$. This graph generalizes the one introduced in \cite{Strong nuevo de JC} for numerical semigroups. While the tree $G(E)$ depends on the fixed monomial order, the set $\mathcal{ST}(E)$ does not. Example \ref{ejemploTree} shows this fact.

Note that $\mathcal{ST}(E)$ may be infinite. To obtain an algorithm, we restrict to semigroups of bounded genus as an auxiliary finiteness condition. Hence, Algorithm~\ref{AlgoritmoStrongTree} computes all elements of $\mathcal{ST}(E)$ up to a given genus.

\begin{algorithm}[H]\label{AlgoritmoStrongTree}
\caption{Computing the set $\{S\in\mathcal{ST}(E)\mid \g(S)\le g\}$.}\label{ComputeST}
\KwIn{An incomparable set $E\subset \N^p$, and a non-zero positive integer $g$.}
\KwOut{All the strong $\CaC(E)$-semigroups with genus at most $g$.}
$B \leftarrow \{R(E)\}$\;
\If{$g< \g (B)$}
    {\Return $\emptyset$\;}
\If{$g= \g (B)$}
    {\Return $\{R(E)\}$\;}
$X\leftarrow B$\;
$i\leftarrow \g (B)$\;

\While {$B\ne\emptyset$ and $i\le g$}{
    $Y \leftarrow \emptyset$\;
    $C\leftarrow B$\;
    \While{$C\neq \emptyset$ }
    {$T \leftarrow \text{First}(C)$\;
    $A \leftarrow \msg(T)\setminus E$\;
    $D \leftarrow \{a\in A \mid a \text{ satisfies Lemma \ref{SmenosXstrong} and }\Fb_E(S)\prec x\}$\;\label{lineaCondiciones}
    $Y \leftarrow Y\cup\{T\setminus\{a\}\mid a\in D\}$\;
    $C \leftarrow C\setminus \{T\}$\;
    }
    $B \leftarrow Y$\;
    $X \leftarrow X\cup Y$\;
    $i \leftarrow i+1$\;
    }
\Return{$X$}
\end{algorithm}

This section concludes by providing an example to illustrate Algorithm \ref{AlgoritmoStrongTree}.

\begin{example}\label{ejemploTree}
Consider the integer cone $\CaC$ generated by $\{(3,2), (7,1)\}$, and let $E\subset \CaC$ be the set $\{(2, 1), (4, 1), (5, 1), (6, 4), (7, 1)\}$, thus $\msg(R(E))=E\sqcup A$ where $A=\{(5, 2), (5, 3), (9, 6), (13, 2)\}$. The set of gaps of $R(E)$ is $\{(3, 1), (3, 2), (6, 1)\}$. Figure \ref{fig:ejemplotree} shows a graphical representation of $R(E)$. The empty circles represent the gaps of $R(E)$, the blue squares denote its minimal generators, and the red circles represent some elements in it.
\begin{figure}[h]
    \centering
    \includegraphics[scale=.45]{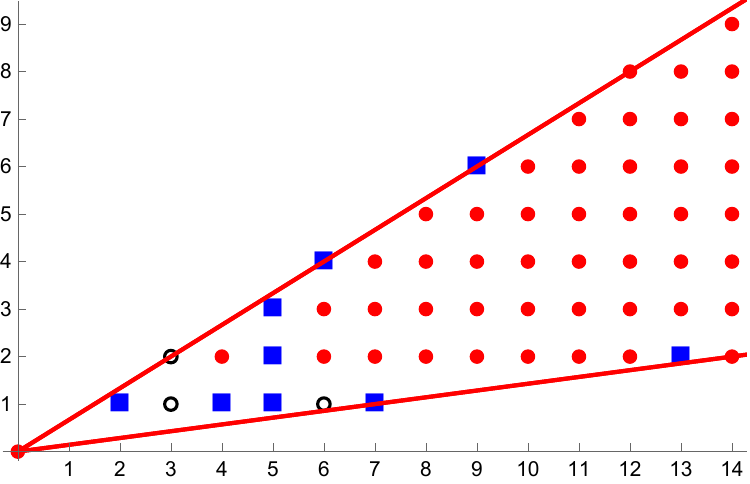}
    \caption{The set $R(\{(2, 1), (4, 1), (5, 1), (6, 4), (7, 1)\})$.}
    \label{fig:ejemplotree}
\end{figure}
Applying Algorithm~\ref{AlgoritmoStrongTree} with $g=6$, we obtain 35 strong semigroups of genus at most $6$.
Figure \ref{fig:ejemplotree1} shows the tree of these strong semigroups with respect to the graded lexicographical order.
Let $\preceq_1$ be the monomial order on $\mathbb{N}^2$ defined as follows. Given $x=(x_1,x_2),\, y=(y_1,y_2)\in \mathbb{N}^2$, we say that
$x \preceq_1 y$ if $x_1+7x_2 < y_1+7y_2$, or if $x_1+7x_2 = y_1+7y_2$ and $x_2 \le y_2$.
Figure \ref{fig:ejemplotree2} corresponds with the tree obtained if the monomial order $\preceq_1$ is considered.
Note that while both trees share the same vertex set, they are not equal.
The root of both trees is $R(E)$, and each vertex is labelled with the corresponding removed minimal generator. For example, in Figure \ref{fig:ejemplotree1}, the rightmost node $(20,3)$ in the last level of the tree is the strong semigroup $R(E) \setminus \{(3,2),(20,3)\}$.

\begin{figure}[h]
    \centering
    \includegraphics[scale=.4]{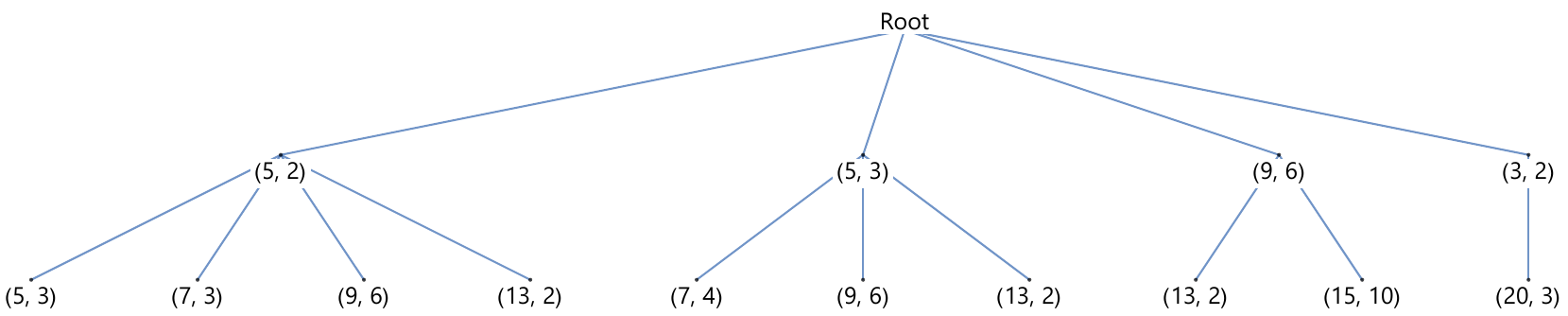}
    \caption{The tree with respect to the graded lexicographical order.}
    \label{fig:ejemplotree1}
\end{figure}
\begin{figure}[h]
    \centering
    \includegraphics[scale=.4]{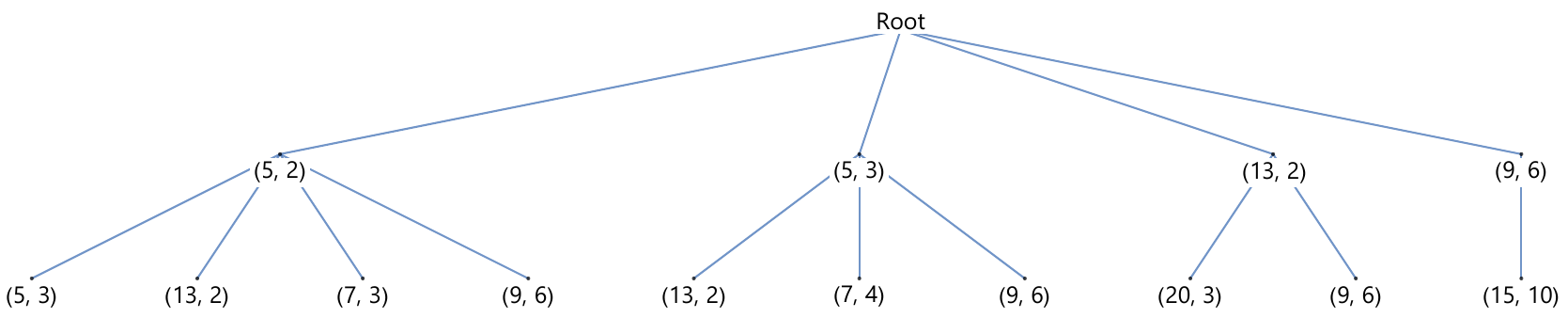}
    \caption{The tree with respect to the monomial order $\preceq_1$.}
    \label{fig:ejemplotree2}
\end{figure}
\end{example}

\section{Finiteness of the set $\mathcal{ST}(E)$}\label{sec:strongMultGenus}
As in the previous section, we fix an incomparable set $E\subset \N^p$. Now, we study the finiteness of the set $\mathcal{ST}(E)$. It is equivalent to characterize when the tree $G(E)$ is finite.
Since $R(E)$ is maximal with respect to the inclusion, there does not exist a strong $\CaC$-semigroup with set of multiplicities $E$ and genus strictly smaller than $\g(R(E))$. Accordingly, throughout this section, we fix an integer $g \ge \g(R(E))$.

If $E=\{m\}\subset \N\setminus\{ 0\}$,  it is shown in \cite{Strong nuevo de JC} that
\begin{multline}\label{NumericalStrong}
L(m,g)=\Big\{0,m,2m, \ldots ,\Big\lfloor \frac{g}{m-1} \Big\rfloor m,\\ \Big\lfloor \frac{g}{m-1} \Big\rfloor m+\big(g\mod (m-1)\big)+1,\to  \Big\},
\end{multline}
is a strong numerical semigroup of genus $g$, where $g\geq m-1$. The symbol $\to$ is used to denote that every integer greater than or equal to $\Big\lfloor \frac{g}{m-1} \Big\rfloor m+\big(g\mod (m-1)\big)+1$ belongs to the set. Hence, there exist infinite strong numerical semigroups with multiplicity $m$.
In contrast, for a non-numerical $\CaC$-semigroup, the set $\mathcal{ST}(E)$ can be finite, as illustrated in Example \ref{EjemploArbolNoInfinito} and Example \ref{EjemploAvacio}.
Inspired by the numerical semigroup \eqref{NumericalStrong}, the next result establishes the existence of strong $\CaC$-semigroups with genus $g \geq g(R(E))$ for certain sets $E$, proving that the set $\mathcal{ST}(E)$ can be infinite.

\begin{proposition}\label{EconAlgunRayoNoCompleto}
Let $E$ be an incomparable set such that $\tau\cap \N^p\neq \tau\cap R(E)$ for some extremal ray $\tau$ of $\CaC(E)$, $n_\tau$ be the minimal generator of $\tau\cap \N^p$, and  $m=min\{k\in\N\setminus\{0\}\mid kn_\tau\in R(E)\}$. Then, for any integer $\tilde{g}\ge m-1$,
\[
T=(R(E)\setminus \tau) \sqcup L_\tau(L(m,\tilde{g}))\in \mathcal{ST}(E),
\]
and $\g(T)=\g(R(E))-(m-1)+\tilde{g}$, where $L_\tau(L(m,\tilde{g}))=\{kn_\tau\mid k\in L(m,\tilde{g})\}$.
\end{proposition}

\begin{proof}
Since $L_\tau(L(m,\tilde{g}))\subset \tau$, then $T=(R(E)\setminus \tau) \sqcup L_\tau(L(m,\tilde{g}))$ is a $\CaC$-semigroup and $\g (T)=\g(R(E))-(m-1)+\tilde{g}$. Note that $E\cap \tau=\{mn_\tau\}$, and $\Mult(T)=E$.

For any $x,y\in T\subset R(E)$ with $x\neq y$, there exist $e\in E$ such that $x+y-e\in R(E)$. If $x+y-e\notin \tau$, then $x+y-e\in T$, and the assertion holds.
Suppose $x+y-e\in \tau$. If $e=m n_\tau$, then $x,y\in \tau$. Hence, $x=k_1 n_\tau$ and $y=k_2 n_\tau$ for some $k_1,k_2\in L(m,\tilde{g})$. Since $L(m,\tilde{g})$ is strong, $k_1+k_2-m\in L(m,\tilde{g})$, and thus,  $x+y-e=(k_1+k_2-m)n_\tau\in L_\tau(L(m,\tilde{g}))$. When $e\neq m n_\tau$, since $x+y-e\in \tau\cap R(E)$, then $x+y-e=mn_\tau+kn_\tau$ for some integer $k\ge 0$. So, $x+y-mn_\tau=e+kn_\tau\in T$.
\end{proof}

To exemplify the previous result, consider the following example.

\begin{example}
Let $\CaC$ and $E$ be the integer cone and the incomparable set given in Example~\ref{ejemploTree}, respectively. We aim obtain a strong $\CaC$-semigroup $T \in\mathcal{ST}(E)$ with $g(T)=5$.
Recall that $\CaH(R(E))=\CaC\setminus R(E)$. Observe that $\tau \cap \CaH(R(E)) \neq \emptyset$ for the extremal ray $\tau=\langle(3,2)\rangle$, as illustrated in Figure~\ref{fig:ejemplotree}. By applying Proposition \ref{EconAlgunRayoNoCompleto}, we obtain $g(T)=3-1+\tilde{g}=5$. Therefore, $L(2,3)=\{0,2,4,6,\to\}$, and consequently $T=R(E)\setminus\{(9,6),(15,10)\}$.
\end{example}

If $\tau \cap \N^p = \tau \cap R(E)$ for every extremal ray $\tau \in \CaC(E)$, then, for a fixed incomparable set $E$, the finiteness of the set $\mathcal{ST}(E)$ remains unsolved. To analyse the above situation, we present some objects.
In \cite{check_C_semigr}, given an affine semigroup $S\subseteq \N^p$, some special lines and sets related to the extremal rays of $\CaC(S)\subseteq \N^p$ are introduced. As previously discussed, the cone $\CaC(S)$ satisfies \eqref{coneHyperplanes}, and let $\Omega$ be the set of extremal rays of $\CaC(S)$. We define
\[
\CaA=\left\{ \sum_{\tau\in \Omega} \lambda_\tau n_\tau \mid 0\le \lambda_\tau\le 1,\mbox{ and } \tau\cap \N^p = \langle n_\tau\rangle \right\}\cap \N^p.
\]
Each extremal ray $\tau\in \Omega$ is determined by $h_\tau^{(1)}(x)=\cdots =h_\tau^{(p-1)}(x)=0$ supported hyperplanes.
Given $\alpha=(\alpha_1, \ldots, \alpha_{p-1})\in \Z^{p-1}$, the $\alpha$-parallel line to $\tau$ given by the solutions of the linear equations $\bigcup _{i=1}^{p-1} \{h_\tau^{(i)}(x)=\alpha_i\}$ is denoted by $\upsilon_{\tau}(\alpha)$. For every integer point $z\in \Z^p$, there exists $\alpha \in \Z^{p-1}$ such that $z$ belongs to $\upsilon_{\tau}(\alpha)$. Note that if $z\in \CaC(S)$, then $\alpha \in \N^{p-1}$. We denote by $\Upsilon _\tau(z)$ the element $(h_\tau^{(1)}(z),\ldots ,h_\tau^{(p-1)}(z))\in\N^{p-1}$ with $z\in\CaC (S)$. Hence, for any $z\in\CaC (S)$, $z\in \upsilon_\tau(\alpha)$ if and only if $\alpha =\Upsilon_\tau(z)$. Consider
$ T_\tau\subset\N^{p-1}$ the affine semigroup generated by the finite set $\{\Upsilon_\tau(z)\mid z\in \CaA\}$ and let $\Gamma_\tau$ be its minimal generating set.
To simplify the notation, we call special line of the cone $\CaC(S)$ to any line $\upsilon_\tau(\alpha)$ with $\alpha \in \Gamma_\tau$ and $\tau \in \Omega$.
Now, we use the aforementioned framework to study the existence of strong $\CaC$-semigroups with genus $g \geq \g(R(E))$.

\begin{proposition}\label{AconLineaMinimal}
Let $E$ be an incomparable set such that $\tau\cap \N^p =  \tau\cap R(E)$ for every extremal ray $\tau$ of $\CaC(E)$, and let $E\sqcup A$ be the minimal generating set of $R(E)$. Assume that there exists $a\in A$ such that $\Upsilon _{\tau '}(a)\in \Gamma _{\tau '}$ for some $\tau'=\langle n \rangle$ extremal ray of $\CaC(E)$. Then,
\[
T=R(E)\setminus\{a,a+n,\ldots , a+\tilde{g}n\}\in \mathcal{ST}(E),
\]
and $g(T)=\tilde{g}+\g(R(E))$ for any $\tilde{g}\in \N$.
\end{proposition}

\begin{proof}
To prove that $T$ is a $\CaC$-semigroup, let us see that $a+in\in \msg(R(E)\setminus\{a,a+n,\ldots ,a+(i-1)n\})$ for all $i\in\{1, 2, \ldots, \tilde{g}\}$. Suppose that $a+in=x+s$ with $x,y\in R(E)\setminus\{a,a+n,\ldots ,a+(i-1)n\}$ non-null. So, $\Upsilon_{\tau'}(a+in)=\Upsilon_{\tau'} (a)= \Upsilon_{\tau'}(x)+\Upsilon_{\tau'}(y)$. The hypothesis $\Upsilon _{\tau '}(a)\in \Gamma _{\tau '}$ means $\Upsilon_{\tau'}(x)=0$ or $\Upsilon_{\tau'}(y)=0$. Without loss of generality, we assume that $\Upsilon_{\tau'}(x)=0$, and  thus $\Upsilon_{\tau'} (a)=\Upsilon_{\tau'}(y)$. Hence, $x=kn$ for some $k\in \N\setminus \{0\}$, and $y=a+\lambda n$ for some $\lambda\ge i$. Since $a+in=x+y$, then $(k+\lambda -i)n=0$, which is not possible.

By definition, an element $e$ belongs to $E$ if and only if $e \in R(E)=T \cup \{a+in \mid 0 \leq i \leq \tilde{g}\}$ and there does not exist $z \in R(E)$ such that $e - z \in \CaC$.
If $e \in T$, then $e \in \Mult(T)$.  If $e \notin T$, then $e = a + in$ for some $0 \leq i \leq \tilde{g}$, which contradicts the minimality of $e$ with respect to $\leq_{\CaC}$. So, $E=\Mult(T)$.

Now, consider two elements $x,y\in T\setminus\{0\}\subseteq R(E)\setminus\{0\}$ such that $x\neq y$.  Since $R(E)$ is strong, there exists $e\in E$ with $x+y-e\in R(E)$. If $x+y-e\in T$, then $T\in \mathcal{ST}(E)$. Otherwise, $x+y-e=a+in$ for some $i\in \{0,1,\ldots, \tilde{g}\}$.
By definition of the set $A$, $a=e'+z$ for some $z\in \CaH(R(E))$ and $e'\in E$. So, $x+y-e=z+in+e'$, and $x+y-e'=z+e+in$. Suppose that $z+e+in\notin T$. Then, $z+e+in=a+jn=e+'z+jn$ for some $j\in\{0,1,, \ldots, \tilde{g}\}$ and $e=(j-i)n+e'$.
Depending on the value of $j-i$, either $e$ or $e'$ is not a minimal generator, a contradiction. We conclude that $x+y-e'=z+e+in\in T$.
\end{proof}

The following example is a direct application of Proposition \ref{AconLineaMinimal}.

\begin{example}
Let $\CaC$ again be the integer cone given in Example~\ref{ejemploTree}, and consider the incomparable set
$$E=\{(3,1), (3,2), (4,1), (4,2), (5,1), (6,1), (7,1)\}.$$
Thus, $R(E)$ is minimally generated by $E\sqcup A$ with $A=\{(5,2), (5,3)\}$, and $\Upsilon_\tau\big((5,3)\big)\in \Gamma_\tau$ where $\tau$ is the extremal ray generated by $(3,2)$. So, $R(E)$ satisfies Proposition \ref{AconLineaMinimal} for $a=(3,2)$, and $T=R(E)\setminus\{a,a+(3,2),a+2(3,2),a+3(3,2)\}\in\mathcal{ST}(E)$ has genus equals four (note that $\Mult(T)=\{(3,1), (3,2), (4,1), (4,2), (5,1), (6,1), (7,1)\}$, and $\msg(T)\setminus\Mult(T)=\{(5,2), (14,9)\}$).
\end{example}

Proposition \ref{AconLineaMinimal} implies that $\mathcal{ST}(E)$ is infinite if its hypotheses hold.
Note that the only remaining case for determining whether $\mathcal{ST}(E)$ is finite occurs when $E$ is an incomparable set such that $\tau \cap \N^p = \tau \cap R(E)$ for every extremal ray $\tau$ of $\CaC(E)$, and there exists no element in $A$ belonging to any special line of $\CaC$.
In general, by \cite[Theorem 9]{check_C_semigr}, an affine semigroup $S$ is a $\CaC$-semigroup if and only if
\begin{itemize}
\item  $(\tau\cap \CaC)\setminus  S$ is finite for every extremal ray $\tau$ of $\CaC$.
\item There exists at least a minimal generator of $S$ belonging to each special line of $\CaC$.
\end{itemize}
So, if $S$ is minimally generated by $E\sqcup A$ such that no elements of $A$ lie on the special lines of $\CaC$, then every special line of $\CaC$ contains an element of $E$.
Moreover, $\langle E\rangle$ is a $\CaC$-semigroup if and only if every extremal ray of $\CaC$ is generated by an element in $E$. Hence, in that case, $\CaC\setminus S\subseteq\CaC\setminus \langle E\rangle$ and then the set of affine semigroups with set of multiplicities $E$ is finite.
In particular, when the set of affine semigroups with set of multiplicities $E$ is finite, $\mathcal{ST}(E)$ is also finite.
The last results of this section characterize these facts.

\begin{theorem}Let $E$ be an incomparable set. The following are equivalent:
\begin{itemize}
\item[{\rm (a)}] $\langle E \rangle$ is a $\CaC$-semigroup.
\item[{\rm (b)}] The set of affine semigroups with set of multiplicities $E$ is finite.
\item[{\rm (c)}] The set of strong affine semigroups with set of multiplicities $E$ is finite.
\end{itemize}
\end{theorem}
\begin{proof}

We begin by (a) $\Rightarrow$ (b). Suppose $\langle E \rangle $ is a $\CaC$-semigroup. So, any affine semigroup $S$ with $\Mult(S)=E$ is also a $\CaC$-semigroup since $\CaC\setminus S\subseteq \CaC\setminus \langle E \rangle$, and we get the result. The implication (b) $\Rightarrow$ (c) is straightforward. To prove (c) $\Rightarrow$ (a), we assume that $\langle E \rangle$ is not a $\CaC$-semigroup and we are going to construct an infinite family of strong affine semigroups with set of multiplicities $E$.
We denote $|x| := x_1+\cdots+x_p$ for $x = (x_1,\ldots,x_p) \in \N^p,$ and consider the family of affine semigroups \[S_m = \langle E \rangle \cup \{x \in R(E) \, \vert \, |x| \geq m\},\] with $m \in \N$. Note that $\Mult(S_m)=E$, for every $m\in \N$.
We denote $M := {\rm max}\{ |e| \, \vert \, e \in E\},$ and let us see that $S_m$ is strong for all $m \geq M$. Take $m \geq M$ and two non-zero elements $x,y \in S_m$. If $x \in \langle E \rangle$, then there exists an $e \in E$ such that $x - e \in \langle E \rangle \subseteq S_m$ and, then, $x+y-e \in  S_m$. If $y \in \langle E \rangle$, we proceed analogously.
If $x, y \in R(E)$ with $|x|,|y|\geq m$, we write $x = e_1 + x'$ and $y = e_2 + y'$ with $e_1, e_2 \in E$ and $x',y' \in \CaC(E)$. Then, $x+y-e_1 = e_2 + x' + y' \in R(E)$ and $|x+y-e_1| = |x| + |y| - |e_1| \geq |x| \geq m$. Thus, $x+y-e_1 \in S_m$.  We conclude that the set of strong affine semigroups with set of multiplicities $E$ is infinite.
\end{proof}

\section{Special strong semigroups}\label{sec:sss}
As in previous sections, consider an incomparable set $E\subset \N^p\setminus\{0\}$. By relaxing the definition of being a strong semigroup, we introduce the concept of a special strong semigroup.

\begin{definition}\label{definitionspecialStrong}
An affine semigroup $S\subseteq\N^p$ with set of multiplicities $E$ is a special strong semigroup if for all $x,y\in S\setminus\{0\}$, there exists $e\in E$ such that $x+y-e\in S$.
\end{definition}

An illustrative example of a special strong semigroup is the set $R(E)$ given in the previous sections. Definition \ref{definitionspecialStrong} is motivated by the notion of being an affine MED-semigroup. First, we recall some alternative definitions of affine MED-semigroups (see \cite{G-T-V25_MED}).

\begin{proposition}\label{prop:MED-NS}
Let $S\subset \N^p$ be an affine semigroup with $E=\Mult (S)$, and let $\Omega$ be the set of extremal rays of $\CaC(S)$. Then, the following statements are equivalent:
\begin{enumerate}
    \item $S$ is a MED-semigroup.
    \item $\Ap(S,E\cap \Omega)= \big(\msg(S)\setminus (E\cap \Omega)\big)\sqcup \{0\}$.
    \item For all $x,y\in S\setminus\{0\},\, x+y-m\in S$ for some $m\in E\cap \Omega$.
\end{enumerate}
\end{proposition}

Any special strong semigroup $S$ can be characterized using its minimal generating set, or the intersection $\Ap (S,E)$, analogously to the previous result.

\begin{proposition}\label{caracterizacion_EMED}
Let $S$ be an affine semigroup. Then $S$ is a special strong semigroup if and only if for any $a_1,a_2\in A$, there exists $e\in E$ such that $a_1+a_2-e\in S$.
\end{proposition}

\begin{proof}
The proof follows directly from the definition of special strong semigroup.
\end{proof}

\begin{proposition}
Let $S\subseteq \N^p$ be an affine semigroup with $E=\Mult (S)$, and $A=\msg(S)\setminus E$. Then, $S$ is a special strong semigroup if and only if $\Ap (S,E)= A \sqcup \{0\}$.
\end{proposition}

\begin{proof}
Assume $S$ is a special strong semigroup. So, it is also a strong semigroup, and by Lemma \ref{lemaAperys}, $\Ap (S,E)= A \sqcup B\sqcup \{0\}$, for some $B\subseteq 2A$. Suppose that there exists $b\in B$, then $b=2a$ with $a\in A$. Since $S$ is a special strong semigroup, $b-e=a+a-e\in S$ for some $e\in E$, and thus $B=\emptyset$. Then, $\Ap (S,E)= A \sqcup \{0\}$.

Now, consider that $\Ap (S,E)= A \sqcup \{0\}$, and suppose $S$ is not a special strong semigroup. Thus, there exist $x,y\in S\setminus\{0\}$ such that $x+y-e\notin S$ for any $e\in E$. Thus, $x+y\in \Ap (S,E)$. It is not possible since $x+y\in A\subset \msg(S)$.
\end{proof}

The proofs of the following results are analogous to the proofs of Lemma 21, Proposition 22 and Theorem 26 in \cite{G-T-V25_MED}, respectively.

\begin{lemma}
Let $S\subseteq \N^p$ be an affine semigroup with $E\subset S$ an incomparable set such that $\CaC(S)=\CaC(E)$. Then, the affine semigroup $T=(E+S)\cup \{0\}$ is a special strong semigroup. Moreover, $$\CaH(T)=\CaH(S)\sqcup \left(\Ap (S,E)\setminus\{0\}\right).$$
\end{lemma}

\begin{corollary}
Under the assumptions of the previous lemma, it holds that:
    \begin{itemize}
        \item $S$ is a $\CaC$-semigroup if and only if $T$ is a $\CaC$-semigroup.
        \item $T=\left(S\setminus \Ap (S,E) \right)\cup \{0\}$.
    \end{itemize}
\end{corollary}

\begin{theorem}
Let $S$ be a semigroup minimally generated by $E\sqcup A$. Then, $S$ is a special strong semigroup if and only if the semigroup $(E+S)\cup \{0\}$ is equal to $S\setminus A$.
\end{theorem}

Since all special strong semigroups are strong semigroups, Lemmas \ref{strongUnionFrob}, \ref{SistMin_S_menos_a}, and \ref{SmenosXstrong} can be adapted to such subfamily, which allow us to define a tree of special strong $\CaC$-semigroups with root $R(E)$, and to introduce an algorithm to compute the special strong $\CaC$-semigroup up to a fixed genus, by replacing the condition ``$a$ satisfies Lemma \ref{SmenosXstrong}" (line \ref{lineaCondiciones}) by ``$a$ satisfies Lemma \ref{SmenosXMED}".

\begin{lemma}\label{MEDUnionFrob}
Let $\preceq$ be a monomial order on $\N^p$, and let $S\neq R(E)$ be a special strong $\CaC$-semigroup, then $S \cup \{\Fb_E(S)\}$ is a special strong $\CaC$-semigroup.
\end{lemma}

\begin{proof}
Following similar arguments as in the proof of Lemma \ref{strongUnionFrob}.
\end{proof}

\begin{lemma}\label{MEDSistMin_S_menos_a}
Let $S=\langle E\sqcup A\rangle\subseteq \N^p$ be an affine special strong semigroup, and $a\in A$. Then, the set $\msg (S\setminus\{a\})\setminus E$ is a subset of $(A\setminus\{a\}) \cup (\{a\}+E)$.
\end{lemma}

\begin{proof}
Since any  special strong semigroup is strong, by Lemma \ref{SistMin_S_menos_a}, the minimal system of generators of $S\setminus\{a\}$ is a subset of $E\cup (A\setminus\{a\})\cup \{2a\} \cup (\{a\}+E)$. Moreover, $2a-e\in S$ for some $e\in S$. Note that $2a-e\neq a$. Thus, $2a$ is not a minimal generator of $S\setminus\{a\}$.
\end{proof}

\begin{lemma}\label{SmenosXMED}
Let $S=\langle E\sqcup A\rangle$ be an affine special strong semigroup, and $a\in A$. The semigroup $S\setminus \{a\}$ is an affine special strong semigroup if and only if: if $a+e= a_1+a_2$ for some $e\in E$, and $ a_1,a_2\in A\setminus \{a\}$, then $a+e-e'\in S$ for some $e'\in E\setminus\{e\}$.
\end{lemma}

\begin{proof}
Trivially, if $S\setminus\{a\}$ is an affine special strong semigroup, and $a+e= a_1+a_2$ for some $e\in E$, and $ a_1,a_2\in A\setminus \{a\}$, then $a_1+a_2-e'\in S\setminus\{a\}$.
Conversely, by Lemma \ref{MEDSistMin_S_menos_a}, the set $A'=\msg (S\setminus\{a\})\setminus E$ is a subset of $(A\setminus\{a\})\cup (\{a\}+E)$, and applying Proposition \ref{caracterizacion_EMED} to prove that $S\setminus\{a\}$ is an affine special strong semigroup, we only need to consider two cases:
\begin{itemize}
\item Let $a_1,a_2\in A\setminus\{a\}$. Since $S$ is a special strong semigroup, $a_1+a_2-e\in S$ for some $e\in E$. If $a_1+a_2-e\neq a$, then $a_1+a_2-e\in S\setminus\{a\}$. Otherwise, by hypothesis $a+e-e'\in S\setminus\{a\}$ for some $e'\in E\setminus\{e\}$.

\item Consider $e,e'\in E$. Hence, $(a+e)+(a+e')-e=2a+e'$ belongs to $S\setminus \{a\}$.

\item Consider $e,e'\in E$ and $a_1\in A\setminus\{a\}$. So, $a_1+(a+e)-e \in S\setminus \{a\}$.
\end{itemize}
\end{proof}

\section{The indispensable monomials of the ideals of strong semigroups}\label{sec:indisp}

We begin this section by introducing several necessary concepts. Given an affine semigroup $S$, the semigroup/toric ideal of $S$ is a polynomial ideal that encodes the arithmetic relations among the generators of $S$ (see \cite{Herzog,MillerSturmfels}). To specialise our notation, we define these concepts from the fixed minimal generating set $E\sqcup A$ of an affine semigroup $S$ in the same way as the previous sections, that is, $E=\{e_1,\ldots ,e_t\}=\Mult (S)$ and $A=\{a_1,\ldots ,a_r\}=\msg(S)\setminus E$.

Let $\k$ be a field, and consider the $S$-graded polynomial ring $R=\k[x_1,\ldots,x_t,y_1,\ldots ,y_r]$ where the $S$-degree of a monomial $X^\alpha Y^\beta = x_1^{\alpha_1}\cdots x_t^{\alpha_t}y_1^{\beta_1}\cdots y_r^{\beta_r}$ is $S\text{-degree}(X^\alpha Y^\beta) = \sum _{i=1}^t\alpha_i e_i + \sum _{i=1}^r\beta_i a_i$. The semigroup ideal of $S$ is the $S$-homogeneous polynomial ideal $I_S\subset R$  defined as
\begin{equation}
I_S=\left\langle X^\alpha Y^\beta - X^\gamma Y^\delta \mid \sum _{i=1}^t\alpha_i e_i + \sum _{i=1}^r\beta_i a_i = \sum _{i=1}^t\gamma_i e_i + \sum _{i=1}^r\delta_i a_i\right\rangle.
\end{equation}
It is well-known that there exist some minimal generating sets for inclusion given by (pure) binomials (see \cite{Herzog}) where
all of them have the same cardinality. They are characterized using simplicial complexes (see \cite{ComplejosSimpliciales}, and the references therein) in the following way. Let $\nabla_m$ be the simplicial complex $\big\{ F \subseteq C_m \mid \gcd(F) \neq 1\big\},$ where $\gcd(F)$ denotes the greatest common divisor of the monomials in $F$, $m$ belongs to $S$, and $C_m=\big\{ X^\alpha Y^\beta  \mid S\mbox{-degree}(X^\alpha Y^\beta)=m\big\}$ (\cite{TesisShalom}). The vertex set of $\nabla_m$ is equivalent to the set of all the ways of writing $m$ as a combination of the minimal generators of $S$.

We have the characterization of the minimal binomial generating sets of $I_S$.
\begin{theorem}(\cite{TesisShalom})\label{TheoremShalom}
Let $\Lambda$ be a minimal binomial generating set of $I_S$,
and $M=\{S\text{-degree}(f)\mid f \in \Lambda\}$. Then, the simplicial complex $\nabla_{m}$ is non-connected if and only if $m\in M$.
\end{theorem}

Note that for any affine semigroup $S$, and every $m\in S$, the simplicial complex $\nabla_m$ only depends on the fixed generating set of $S$. Thus, the set of $S$-degrees of any two different minimal generating sets of the ideal $I_S$ is equal. These $S$-degrees are known as the Betti elements of $S$ (see \cite{G-O-10}). A monomial is called an indispensable monomial when it appears in any minimal generating set of $I_S$. Equivalently, a monomial $X^\alpha Y^\beta$ of $S\text{-degree } m$ is indispensable if and only if $\{X^\alpha Y^\beta\}$ is a connected component of the non-connected simplicial complex $\nabla_{m}$ (see, e.g., \cite{CKT}).

From Corollary \ref{escritura_unica}, the simplicial complex associated with an element belonging to the Apéry set is easily determined in this lemma.

\begin{lemma}
Let $S=\langle E\sqcup A\rangle$ be an affine (special) strong semigroup. Then, $\nabla _m$ has only a vertex for every $m\in \Ap (S,E)$.
\end{lemma}

By Theorem \ref{TheoremShalom}, the knowledge of the simplicial complexes $\nabla_m$ associated to the $S$-degree $m$ allows us to obtain the Betti set of the polynomial ideal $I_S$. We consider this approach to determine the indispensable monomials for some strong affine semigroups. To do it, new definitions are necessary. For any $j\in[r]$, let
$$T_j=\Big\{ \alpha \in \N^t\mid \sharp C_{\sum_{i=1}^t \alpha_ie_i+a_j}\ge 2, \mbox{ and } \sharp C_{\sum_{i=1}^t \alpha_ie_i}=1\Big\},$$
and
$$T_j'=\Big\{ \alpha \in \N^t\mid \sharp C_{\sum_{i=1}^t \alpha_ie_i+2a_j}\ge 2, \mbox{ and } \sharp C_{\sum_{i=1}^t \alpha_ie_i+a_j}=1\Big\}.$$
For any non-zero integer $l\in \N$, and any set $L\subseteq \N^l$, ${\rm Hilb}(L)$ denotes the Hilbert basis of $L$. This finite set corresponds to the minimal elements in $L$ with respect to the componentwise partial order on $\N^l$.

The following theorem characterizes some indispensable monomials of strong affine semigroups, and it is the main key to obtain the following results of this section.

\begin{theorem}\label{TheoremYiYj}
Let $S$ be a strong affine semigroup minimally generated by $E\sqcup A$ with $A\neq \emptyset$. Then, the following monomials are indispensable:
\begin{enumerate}
\item $y_i y_j$ for $i,j\in [r]$ with $i\neq j$.
\item $y_i^2$ for any $i\in [r]$ such that $2a_i\notin \Ap(S,E)$.
\item $y_i^3$ for any $i\in [r]$ such that $2a_i\in \Ap(S,E)$.
\item $X^\omega y_j$  for any $j\in [r]$ and $\omega \in {\rm Hilb} (T_j)$.
\item $X^\omega y_j^2$  for any $j\in [r]$ with $2a_j\in\Ap(S,E)$, and $\omega \in {\rm Hilb} (T'_j)$.
\end{enumerate}
Moreover, they are the unique indispensable monomials involving at least one variable in $\{y_1,\ldots, y_r\}$.
\end{theorem}

\begin{proof}
Let $M$ be a monomial and denote $m = S\text{-degree}(M)$. We have that $M$ is indispensable if and only if $\nabla_m$ is not connected and $\{M\}$ is a connected component. Equivalently, $M$ is indispensable if and only if the cardinality of $C_{m}$ is at least two, and for any monomial $M' \in C_m$ distinct from $M$, then $\gcd(M, M') = 1$.

Consider $i,j\in [r]$ with $i\neq j$. Since $S$ is a strong semigroup, the cardinality of $C_{a_i+a_j}$ is at least two, let $X^\alpha Y^\beta \in C_{a_i+a_j}\setminus\{y_iy_j\}$. Note that $i,j\notin \supp (\beta)$. Otherwise, $a_i$ or $a_j$ is not a minimal generator of $S$, or $S$ is such that $S\cap (-S)\neq \{0\}$. Hence, $\gcd(X^\alpha Y^\beta,y_iy_j) = 1$.
Analogously, for any $2a_i\notin\Ap(S,E)$, there also exists more than one monomial in $C_{2a_i}$, but there is no one as $y_iX^{\alpha'} Y^{\beta'}$ in $C_{2a_i}\setminus\{y_i^2\}$.
For every $y_i^3$ such that $2a_i \in \Ap (S,E)$, we know that $\sharp C_{3a_i} \ge 2$. Since $C_{2a_i}=\{y_i^2\}$, there are no monomials $y_iX^\alpha Y^\beta$ in $C_{3a_i}$ with $\alpha\neq 0$.
To prove the fourth item, we consider an integer $j\in [r]$ and an element $\omega \in {\rm Hilb} (T_j)$. Hence, $\sharp C_{\sum_{i=1}^t \omega_i e_i+a_j}\ge 2$, and $\sharp C_{\sum_{i=1}^t \omega_i a_i}=1$. Let $X^\alpha Y^\beta$ be a monomial in $C_{\sum_{i=1}^t \omega_i e_i+a_j}\setminus \{X^\omega y_j\}$. If $j \in \supp(\beta)$, then $\{X^\omega,X^\alpha \frac{Y^\beta}{y_j}\}\subseteq C_{\sum_{i=1}^t \omega_i e_i}$, but this contradicts with the cardinality of $C_{\sum_{i=1}^t \omega_i e_i}$. Furthermore, if we assume that there exists some $k\in \supp(\omega)\cap \supp(\alpha)$, then $(\omega_1,\ldots, \omega_{k-1},\omega_{k}-1,\omega_{k+1},\ldots \omega_t)$ belongs to $T_j$, but this fact implies that $\omega \notin {\rm Hilb} (T_j)$, which is not possible. Consequently, for every $X^\alpha Y^\beta \in C_{\sum_{i=1}^t \omega_i e_i+a_j}\setminus \{X^\omega y_j\}$, we have that $\gcd(X^\omega y_j,X^\alpha Y^\beta) = 1$. The last item can be proved similarly, taking $\omega \in {\rm Hilb} (T'_j)$ and assuming that $2a_j\in\Ap(S,E)$. Hence, the monomials appearing in the different items are indispensable.

Let us prove now that there are no other indispensable monomials. We consider $X^\alpha Y^\beta \in C_m$ an indispensable monomial  involving at least one variable in $\{y_1,\ldots,y_r\}$. We study different cases based on the cardinality of the support of $\beta=(\beta_1,\ldots,\beta_r)$:

\noindent {\em Case 1:} $\sharp \supp (\beta)\ge 2$. Without loss of generality, we can assume that $\beta_1\beta_2\neq 0$. Since $S$ is a strong semigroup, there exist $k\in [t]$, $\alpha'\in \N^t$, and $\beta'\in \N^r$ such that the monomial $M := x_kX^{\alpha+ \alpha'}\frac{Y^{\beta + \beta'}}{y_1y_2} \in C_m$. As $X^\alpha Y^\beta$ is indispensable, then $1 = \gcd(X^\alpha Y^\beta, M).$ Thus, $X^\alpha Y^\beta = y_1 y_2$.

\noindent {\em Case 2:} If $\sharp \supp (\beta) = 1$. Without loss of generality, we can assume that $\beta_1 \neq 0.$

{\em Case 2.1}: assuming that $2a_1 \in \Ap(E,S)$, we distinguish the following cases,

\begin{itemize}
\item If $\beta_1 \geq 3$, then  there exist $k\in [t]$, $\alpha'\in \N^t$, and $\beta'\in \N^r$ such that $3 a_1 = e_k + \sum_{i = 1}^t \alpha_i' e_i + \sum_{i = 1}^r \beta_i' a_i$. Hence, $M = x_k X^{\alpha + \alpha'} y_1^{\beta_1 - 3} Y^{\beta'}\in C_m$. Since $X^\alpha Y^{\beta}$ is indispensable, this means that $1 = \gcd(X^\alpha Y^\beta, M) = X^{\alpha}.$ Thus, $X^\alpha Y^\beta = y_1^3$.

\item  If $\beta_1 = 2$, let us prove that $\alpha \in {\rm Hilb} (T'_j)$. Assume otherwise that $\alpha \notin {\rm Hilb} (T'_j)$, this means that either: (i) $\alpha \notin T'_j$, or (ii) there exists $\alpha' \in T'_j$ such that $X^{\alpha'}$ divides $X^{\alpha}$. In (i), the set $C_{m-a_1}$ has two different elements $N_1$ and $N_2$. So, $N_1 y_1, N_2 y_1 \in C_{m}$ and $\gcd(N_iy_1,X^\alpha y_1^2) \neq 1$ for $i = 1,2$; but this contradicts that $X^{\alpha} Y^{\beta}$ is indispensable. In (ii), the set $C_{\sum_{i = 1}^t \alpha_i'e_i + 2 a_1}$ has at least two elements $N_1, N_2$. So  $X^{\alpha - \alpha'}N_1, X^{\alpha - \alpha'} N_2 \in C_{m}$ and $\gcd(X^{\alpha - \alpha'} N_i,X^\alpha y_1^2) \neq 1$ for $i = 1,2$; but this also contradicts that $X^{\alpha} Y^{\beta}$ is indispensable.

\item  If $\beta_1 = 1$, one can prove that $\alpha \in {\rm Hilb} (T_j)$ analogously.
\end{itemize}

{\em Case 2.2}: when $2a_1 \notin \Ap(E,S)$, and proceeding as in {\em Case 2.1}, one can prove that if $X^{\alpha}Y^{\beta}$ is indispensable, then $\beta_1 = 1$ and $\alpha \in {\rm Hilb} (T_j)$.
\end{proof}

\begin{corollary}\label{CorollaryYiYj}
Given a special strong affine semigroup  $S$ minimally generated by $E\sqcup A$ with $A\neq \emptyset$. Then, the following monomials are indispensable:
\begin{enumerate}
\item $y_i y_j$ for $i,j\in [r]$.
\item  $X^\omega y_j$  for any $j\in [r]$ and $\omega \in {\rm Hilb} (T_j)$.
\end{enumerate}
Moreover, they are the unique indispensable monomials involving at least one variable in $\{y_1,\ldots, y_r\}$.
\end{corollary}

From now on, we consider that the affine semigroup $S\subseteq \N^p$ is simplicial, that is, the non-negative integer cone generated by $S$ has dimension $p$ and just $p$ extremal rays. For a family of this kind of semigroups, Theorem \ref{TheoremYiYj} can be improved.

\begin{lemma}\label{IndispSimplicial}
Let $S\subseteq \N^p$ be a simplicial strong semigroup with $\sharp E= p$ and $A$ non-empty. Then, the set of $S$-degrees of the indispensable monomials of $I_S$ is
\begin{multline*}
    \{a+a'\mid a,a'\in A,\, a\neq a'\}\cup \{2a\mid 2a\in 2A\setminus \Ap (S,E)\} \\\cup \{3a\mid 2a\in 2A\cap\Ap (S,E)\}\\ \cup \big\{\sum_{i=1}^t\omega_i e_i+a_j \mid j\in [r]\mbox{ and } \omega \in {\rm Hilb} (T_j)\big\}
    \\ \cup \big\{\sum_{i=1}^t\omega_i e_i+2a_j \mid j\in [r],\, 2a_j\in\Ap(S,E) \mbox{ and } \omega \in {\rm Hilb} (T'_j)\big\}.
\end{multline*}

\end{lemma}

\begin{proof}
Since $S\subseteq \N^p$ is a simplicial strong semigroup, it is isomorphic to a strong affine semigroup $S'\subseteq \N^p$ minimally generated by $E'\sqcup A'$, with  $E'=\min_{\le_{\N^p}}(S'\setminus \{0\})=\cup_{i=1}^p\{d_i b_i\}$ where $\{b_1,\ldots ,b_p\}$ is the usual canonical basis of $\N^p$, and $d_i$ is a non-zero natural number. In particular, note that $E'$ corresponds with the image set of $E$, and $A'$ with the image set of $A$. Moreover, $I_S=I_{S'}$. Moreover, there is no minimal generator of $I_S$ in $\k[X]$.
Suppose that $f=x_{i_1}^{\alpha_{i_1}}\cdots x_{i_h}^{\alpha_{i_h}} -x_{i_{h+1}}^{\alpha_{i_{h+1}}}\cdots x_{i_t}^{\alpha_{i_t}}Y^\gamma$ is a minimal generator of $I_S$, thus these monomials belong to different connected components of the associated simplicial complex. Note that $S'\text{-degree}(f) = \sum_{j = 1}^h \alpha_{i_j}b_{i_j}$ and, then,  $\alpha_{i_{h+1}}=\cdots =\alpha_{i_{t}} =0$ since $E'=\cup_{i=1}^p\{d_i b_i\}$, and that $\gamma$ has to be non-zero.
Now, we prove that, if $f$ is a minimal generator, then $Y^\gamma$ is also an indispensable monomial.
We distinguish cases depending on the cardinality of $\supp(\gamma)$.
\begin{itemize}
    \item If $\sharp\supp(\gamma)\geq 3$, then applying the definition of being strong  $x_{i_1}^{\alpha_{i_1}}\cdots x_{i_h}^{\alpha_{i_h}}$ and $Y^\gamma$ belong to the same connected component, which is impossible.
    \item If $\sharp\supp(\gamma)= 2$, and $\sum _{i=1}^r \gamma_i\geq 3$, then the monomials are also in the same connected component, which is false. So, $\sum _{i=1}^r \gamma_i=2$, and the $S$-degree of the binomial is $a+a'$ for some $a,a'\in A$ such that $a\neq a'$.
    \item Assume $\sharp\supp(\gamma)= 1$. Without loss of generality, we may assume that $\gamma_1 > 0$. If $\gamma_1\geq 4$, we obtain that both monomials belong to the same connected component, a contradiction. If $\gamma_1= 3$, then the  $S$-degree of $y_1^3$ is $2a_1+a_1$ with $a_1\in A$ and we distinguish two situations depending on whether $2a_1$ belongs to $\Ap(S,E)$.
    If $2a_1\notin \Ap(S,E)$, then we conclude that both monomials are in the same connected component. Otherwise, the result holds. Finally, if $\gamma_1=2$, then $2a_1\notin\Ap(S, E)$, which completes the proof.
\end{itemize}
\end{proof}

The previous results allow us to determine all the Betti elements of the ideal $I_S$ when the ideal is Cohen-Macaulay. In general, given $R$ a Noetherian local ring, a finite $R$-module $M\neq0$ is a Cohen-Macaulay module if the depth of $M$ is equal to its dimension.
If $R$ is a Cohen-Macaulay module, then $ R$ is called a Cohen-Macaulay ring (see \cite{libro-C-M} for details). A semigroup is called Cohen-Macaulay if its associated ring $\k[S]$ is Cohen-Macaulay.

\begin{proposition}\label{CM_simplicial_strong}
Let $S\subseteq \N^p$ be a Cohen-Macaulay simplicial strong semigroup with $\sharp E= p$ and $A$ a non-empty set. Then, the set of Betti elements is
\begin{multline*}
    \{a+a'\mid a,a'\in A,\, a\neq a'\}\cup \{2a\mid 2a\in 2A\setminus \Ap (S,E)\} \\\cup \{3a\mid 2a\in 2A\cap\Ap (S,E)\}.
\end{multline*}
\end{proposition}

\begin{proof}
First, we prove there is no minimal generator in $I_S$ with both monomials in $\k[X,Y]\setminus \big(\k[X]\cup \k[Y]\big)$. Otherwise, we can assume that $X^{\alpha}Y^{\beta}-X^{\alpha'}Y^{\beta'}$ is a minimal generator of $I_S$ with $\alpha\neq 0 \neq \alpha'$ and $\beta\neq 0 \neq \beta'$. Without loss of generality, we assume that $\alpha_1$ and $\alpha_p'$ are non-null.
Hence, $(\alpha_1-1) e_1+\sum_{k=2}^p \alpha_ke_k+ \sum _{q=1}^r\beta_q a_q+e_1= (\alpha'_p-1) e_p+\sum_{k=1}^{p-1} \alpha'_ke_k+\sum _{q=1}^r\beta'_q a_q+e_p$. Since $S$ is Cohen-Macaulay, by \cite[Theorem 2.6]{Goto}, the element $(\alpha_1-1) e_1+\sum_{k=2}^p \alpha_ke_k+ \sum _{q=1}^r\beta_q a_q-e_p= (\alpha'_p-1) e_p+\sum_{k=1}^{p-1} \alpha'_ke_k+\sum _{q=1}^r\beta'_q a_q-e_1$ belongs to $S$. Hence, the monomials $X^{\alpha}Y^{\beta}$ and $X^{\alpha'}Y^{\beta'}$ are in the same connected component of its associated simplicial complex, which is not possible.

By the proof of Lemma \ref{IndispSimplicial}, we know that the Betti element obtained from any minimal binomial in $I_S$ with some monomial in $\k[X]$ corresponds to the $S$-degree of some indispensable monomial in $\k[Y]$.

Let $Y^\gamma-X^\alpha Y^\beta$ be a minimal generator of $I_S$ and suppose that $Y^\gamma$ is not an indispensable monomial and $\beta\neq 0$. So, $\sum_{i=1}^r\gamma_i \geq 3$, $Y^\gamma=y_i^3$ with $2a_i\notin \Ap(S,E)$, or $Y^\gamma=y_i^2$ with $2a_i\in \Ap(S,E)$. The last possibility does not happen since the binomial is non-zero.
The first two possibilities imply there is a monomial $X^{\alpha'}Y^{\beta'}$ in the same connected component of $Y^\gamma$ with $\alpha'$, and $\beta'$ non-null. If $\alpha\neq 0$, then $X^{\alpha'}Y^{\beta'}-X^\alpha Y^\beta$ is a minimal generator, but it is a contradiction. Assuming $\alpha=0$, if $Y^\beta$ is indispensable, then the result holds. Otherwise, analogously to $Y^\gamma$, we prove that there exists a binomial $X^{\alpha'}Y^{\beta'}-X^{\alpha''} Y^{\beta''}$ which is a minimal generator of $I_S$, and it is a contradiction again.
\end{proof}

The last example illustrates Proposition \ref{CM_simplicial_strong}.

\begin{example}
Let $S\subset \N^2$ be the Cohen-Macaulay simplicial strong semigroup minimally generated by $E\cup A$, where $E= \{(1,2),(2,0)\}$ and $A=\{(4,3),(5,4),(6,5)\}$. This is the convex body semigroup given by the rational triangle with vertex set $\{(1, 2), (2, 0), (1.2, 0.9)\}$. By Corollary \cite[Corollary 12]{G-V-2014_CohenM}, the affine semigroup $S$ is Cohen-Macaulay.

In this example, the ideal $I_S\subset \k[x_1,x_2,y_1,y_2,y_3]$ is minimally generated by
\begin{multline*}\Lambda=\{x_1x_2y_2-y_1^2, x_1x_2y_3-y_1y_2, y_1y_3-y_2^2, x_1^4x_2^3-y_1y_3, \\ x_1^3x_2^2y_1-y_2y_3, x_1^2x_2y_1^2-y_3^2\}.
\end{multline*}
Note that the $S$-degrees of the elements in $\Lambda$ correspond to the Betti set appearing in Proposition \ref{CM_simplicial_strong}.
\end{example}

\subsection*{Funding}

The first author is partially supported by grant
PID2023-149508NB-I00 funded by MICIU/AEI /10.13039/501100011033 and by
FEDER, UE.

The second author is partially supported by the Centro de Estudios Avanzados en Física, Matemáticas y Computación (CEAFMC) of the University of Huelva. The author was also partially supported by the Plan Propio de Estímulo y Apoyo a la Investigación y Transferencia 2025–2027 (Universidad de Cádiz).

The last two authors are partially supported by grant PID2022-138906NB-C21 funded by MICIU/AEI/ 10.13039/501100011033 and by ERDF/EU, and by Junta de Andalucía research group FQM343.

This publication and research have been partially granted by INDESS (Research University Institute for Sustainable Social Development), Universidad de Cádiz, Spain.

\subsection*{Author information}

I. García-Marco. Instituto de Matemáticas y Aplicaciones (IMAULL), Universidad de La Laguna, E-38200 La Laguna (Spain). E-mail: iggarcia@ull.edu.es.

\medskip

\noindent
R. Tapia-Ramos. Departamento de Ciencias Integradas y Centro de Estudios Avanzados en Física, Matemáticas y Computación, Universidad de Huelva, E-21071  Huelva (Spain)
E-mail: raquel.tapia@dci.uhu.es.

\medskip

\noindent
A. Vigneron-Tenorio. Departamento de Matem\'aticas/INDESS (Instituto Universitario para el Desarrollo Social Sostenible), Universidad de C\'adiz, E-11406 Jerez de la Frontera (C\'{a}diz, Spain).
E-mail: alberto.vigneron@uca.es.

\end{document}